\documentclass[12pt]{article}
\usepackage{amsmath,amssymb,amsthm}
\usepackage{enumerate}
\usepackage{graphicx}
\usepackage[T1]{fontenc}
\usepackage{caption}
\usepackage{cite}

\def\draw #1 by #2 (#3){
	\vbox to #2{
		\hrule width #1 height 0pt depth 0pt
		\vfill
		\special{picture #3} 
	}
}

\def\scaleddraw #1 by #2 (#3 scaled #4){{
		\dimen0=#1 \dimen1=#2
		\divide\dimen0 by 1000 \multiply\dimen0 by #4
		\divide\dimen1 by 1000 \multiply\dimen1 by #4
		\draw \dimen0 by \dimen1 (#3 scaled #4)}
}

\newtheorem{theorem}{Theorem}[section]
\newtheorem{example}[theorem]{Example}
\newtheorem{problem}[theorem]{Problem}
\newtheorem{defin}[theorem]{Definition}
\newtheorem{lemma}[theorem]{Lemma}

\usepackage{xcolor}
\usepackage{float}
\usepackage{graphicx,subcaption,lipsum}
\usepackage{mathtools}

\newtheorem{nt}{Note}

\newcommand{\singlespacing}{\let\CS=\@currsize\renewcommand{\baselinestretch}{1}\tiny\CS}
\newcommand{\oneandahalfspacing}{\let\CS=\@currsize\renewcommand{\baselinestretch}{1.25}\tiny\CS}
\newcommand{\doublespacing}{\let\CS=\@currsize\renewcommand{\baselinestretch}{1.35}\tiny\CS}

\newtheorem{rule-def}[theorem]{Rule}

\renewcommand{\baselinestretch}{1.5}
\numberwithin{equation}{section}
\begin{document}
	\baselineskip 16pt
	
	\newcommand{\la}{\lambda}
	\newcommand{\si}{\sigma}
	\newcommand{\ol}{1-\lambda}
	\newcommand{\be}{\begin{equation}}
		\newcommand{\ee}{\end{equation}}
	\newcommand{\bea}{\begin{eqnarray}}
		\newcommand{\eea}{\end{eqnarray}}
     \begin{center}   {\Large\textbf{Extremal $ABS$ Spectral Radius in Bicyclic and Bipartite Unicyclic Graphs}}
     	\\
     	\vspace{8mm}
     	
     	{\large \bf 
     	Swathi Shetty$^{1}$, B. R. Rakshith$^{*,2}$, Sayinath Udupa N. V. $^{3}$}\\
     	\vspace{6mm}
     	
     	\baselineskip=0.20in
     	
     	{\it
     	Department of Mathematics, Manipal Institute of Technology\\ Manipal Academy of Higher Education\\ Manipal 576104, India.}\\
     {\rm E-mail:} {\tt
     	swathi.dscmpl2022@learner.manipal.edu$^{1}$\\
     	 rakshith.br@manipal.edu$^{*,2}$\\
     	sayinath.udupa@manipal.edu	$^{3}$.}
     \end{center}
 \[\text{September 10, 2025}^\dagger\]
 \footnotetext{*Corresponding author; $\dagger$ Initial journal submission date}
 
     \begin{abstract}
    The $ABS$ spectral radius of a graph 
    $G$ is defined as the largest eigenvalue of its $ABS$ matrix. Motivated by recent studies on this parameter, in this paper, we determine the bipartite unicyclic graphs that attain the largest $ABS$ spectral radius. Furthermore, we characterize the bicyclic graphs that attain the largest and the second largest $ABS$ spectral radii.
    
      
\end{abstract}
  \section{Introduction}
Let $G$ be a simple undirected graph  with  vertex
set $V (G) = \left\{v_1, v_2,\dots , v_n\right\}$ and edge set $E(G)$. As usual, we denote the path graph and the cycle graph on $n$ vertices, by $P_{n}$ and $C_{n}$, respectively.  The well-known adjacency matrix of $G$ is denoted by $A(G)$, and its $i$-th largest eigenvalue is given by $\lambda_{i}$. In recent years, extensive research has been carried on  the spectral properties of topological matrices. The first Zagreb matrix, the $ABC$ matrix, and the Sombor matrix are among the recently introduced degree-based topological matrices that have been extensively investigated in the literature (see \cite{li2020abc,das2024first,mei2023extreme,tabassum2023relationship,liu2021spectral} for details). The atom-bond sum-connectivity index of a graph $G$, denoted by $ABS(G)$, is a degree-based topological index defined as \[ABS(G)=\displaystyle\sum_{v_{i} v_{j}\in E(G)}\sqrt{1-\dfrac{2}{d_{i}+d_{j}}}.\] The topological matrix associated with $ABS(G)$ is called the atom-bond sum-connectivity matrix of $G$ and is denoted by $\mathcal{ABS}(G)$. Its $(i,j)$-th entry is given by \[(\mathcal{ABS}(G))_{ij}=\left\{\begin{array}{cc}
	\sqrt{1-\dfrac{2}{d_{i}+d_{j}}},& \text{if}\,\, v_{i}v_{j}\in E(G)\\
	0, &otherwise
\end{array}.\right.\] 
 The $i$-th largest eigenvalue of $\mathcal{ABS}(G)$ is denoted by   $\eta_i(G)$. In particular, $\eta_1$ is 
the spectral radius of $\mathcal{ABS}(G)$, hereafter referred to as the $ABS$ spectral radius of $G$. The matrix $\mathcal{ABS}(G)$
was first introduced in~\cite{lin2024abs}, where the authors investigated the properties of the $ABS$ Estrada index, $\sum_{i=1}^{n}e^{\eta_{i}}$ and discussed its chemical significance. The eigenvalue properties of the $ABS$ matrix were further explored in \cite{shetty2024mathcal}.\\

Recently, extremal problems concerning the $ABS$ spectral radius of trees and unicyclic graphs have been studied. In \cite{lin2024abstrees}, Lin et al.  characterized trees on $n$ vertices with largest $ABS$ spectral radius and demonstrated the chemical importance of $ABS$ spectral radius. Moreover, they posed a general problem of characterizing graphs with maximum $ABS$ spectral radius in a given class of graphs. In \cite{shetty2024mathcal1}, the authors have identified unicyclic graphs on $n$ vertices that attains the maximum $ABS$ spectral radius. Motivated by these works, in Section 2,  we determine the bipartite unicyclic graphs that attain the largest $ABS$ spectral radius. In section 3, we characterize the bicyclic graphs that attain the largest and the second largest $ABS$ spectral radii.
  \section{$ABS$ spectral radius of bipartite unicyclic graphs}\label{bug}
 In this section, we identify bipartite unicyclic graphs with largest $ABS$ spectral radius. The class of all bipartite unicyclic graphs of order $n$ is denoted as $\mathcal{U}_n^{b}$.  
 Let $\mathcal{H}_1$ be the unicyclic graph of order $n\ge 5$ obtained by attaching $n-4$ pendant vertices to one vertex of the cycle $C_4$. The graph $\mathcal{H}_2$ is obtained by attaching $n-5$ pendant vertices to one vertex of $C_{4}$, and then attaching a new pendant vertex to another vertex of $C_{4}$. See Fig. \ref{girth4}. 
\begin{figure}[H]
~~~~~~~~~~~~\includegraphics[width=0.75\textwidth]{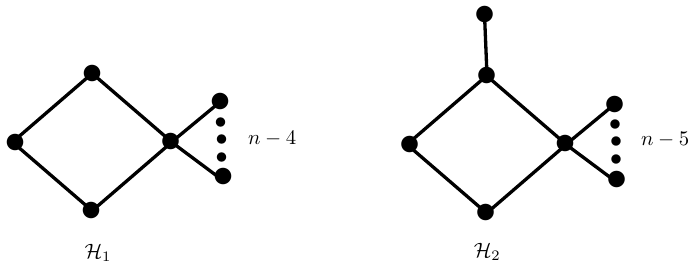}
\caption{The graphs $\mathcal{H}_1$ and $\mathcal{H}_2$.}
\label{girth4}
\end{figure}  
\noindent
 Let $M=(m_{ij})_{p\times q}$ and $N=(n_{ij})_{p\times q}$. We write $M\preceq N$ if $m_{ij}\le n_{ij}$ for all $i$, $j$. 
 The following lemmas are essential to prove our main results. 

\begin{lemma}{\rm\cite{horn2012matrix}}\label{b_1b_2}
	{Let $M, N$ be non-negative matrices of order $n$ with spectral radius $\rho(M)$ and $\rho(N)$, respectively. If $M\preceq N$, then $\rho(M)\le \rho(N)$. Further, if $M$ is irreducible and $M\neq N$, then $\rho(M)< \rho(N)$.}\end{lemma}  
\begin{lemma}{\rm\cite{nath2011second}}\label{girth41} Let $\mathcal{H}_2$ be the graph as depicted in Fig. \ref{girth4}. Then for any graph $G$ in the class  $\mathcal{U}_n^{b}\backslash\left\{\mathcal{H}_1\right\}$, $\lambda_1(G)\le \lambda_1(\mathcal{H}_2)$. Equality holds if and only if $G\cong \mathcal{H}_2$.
\end{lemma}
\begin{lemma}\label{unblemma1}
If $G\in \mathcal{U}_n^{b}\backslash\left\{{\mathcal{H}_1}\right\}$, then $\eta_1(G)<\dfrac{n-2}{\sqrt{n}}$ for all $n\ge 10$. 
\end{lemma}
\begin{proof}
Since $d_i+d_j\le n$, for every edge $v_iv_j\in E(G)$,  $(\mathcal{ABS}(G))_{ij}\le \sqrt{\dfrac{n-2}{n}}$.  Therefore, $\mathcal{ABS}(G)\preceq \sqrt{\dfrac{n-2}{n}} A(G)$, and so by Lemma \ref{b_1b_2},  $\eta_1(G)< \sqrt{\dfrac{n-2}{n}}\lambda_1(G)$. Consequently, by Lemma \ref{girth41}, $\eta_{1}(G)<\sqrt{\dfrac{n-2}{n}}\lambda_{1}(\mathcal{H}_2)$.\\[2mm]
  Claim: $\lambda_1(\mathcal{H}_2)<\sqrt{n-2}$. The adjacency matrix of $\mathcal{H}_2$ is as follows.  
   $$ \begin{bmatrix}
	0 &1&0&1&0&1&\ldots&1\\
	1&0&1&0&0&0&\ldots&0\\
    0&1&0&1&0&0&\ldots&0\\
	1&0&1&0&1&0&\ldots& 0\\
    0&0&0&1&0&0&\ldots&0\\
    1&0&0&0&0&0&\ldots& 0\\
	\vdots&&&&&\vdots&\ddots&\vdots\\
    1&0&0&0&0&0&\ldots& 0
\end{bmatrix}.$$  
Therefore, the characteristic polynomial of $A(\mathcal{H}_2)$ is $x^{n-6}\zeta(x)$, where $\zeta(x):=x^6-nx^4+(3n-13)x^2+5-n$. It is easy to verify that, $\zeta\left(\sqrt{n-2}\right)>0$ for all $n\ge 10$. Let $\zeta_1(x)=\zeta\left(x+\sqrt{n-2}\right)$. Then $\zeta_1'(x)=2(3n-13)\left(x+\sqrt{n-2}\right)-4n\left(x+\sqrt{n-2}\right)^3+6\left(x+\sqrt{n-2}\right)^5>0$ for all $x\ge 0$. Thus, $\zeta(x)$ is increasing for $x\ge \sqrt{n-2}$. Since $\zeta\left(\sqrt{n-2}\right)>0$ for all $n\ge 10$, it follows that $\zeta(x)$ has no root in the interval $\left[\sqrt{n-2},\infty\right)$. Proving that $\lambda_1(\mathcal{H}_2)<\sqrt{n-2}$. Hence
 $\eta_1(G)<\dfrac{n-2}{\sqrt{n}}$.
\end{proof}
The following result gives the spectral radius of $\mathcal{ABS}(\mathcal{H}_{1})$. The proof is omitted, since the roots of the characteristic polynomial of  $\mathcal{ABS}(\mathcal{H}_{1})$ can be listed explicitly. 
\begin{lemma}\label{un4lemma2}
For $n\ge 5$,
\[\eta_1^2(\mathcal{H}_1)=\dfrac{n^3-4n^2+5n+4+\sqrt{n^6-12n^5+58n^4-108n^3+41n^2+40n+16}}{2n(n-1)}.\]
\end{lemma}
\begin{theorem}
Among all bipartite unicyclic graphs of order $n\ge 5$, the graph $\mathcal{H}_1$  attains the maximum $ABS$ spectral radius.
\end{theorem}
\begin{proof}
For $5\le n\le 9$, it can be verified using Sage \cite{Sage} that $\mathcal{H}_1$ has the maximum $ABS$ spectral radius.
Let $n\ge 10$. By Lemma~\ref{unblemma1}, for any $G\in \mathcal{U}_n^{b}\backslash\left\{\mathcal{H}_1\right\}$, $\eta_1(G)<\dfrac{n-2}{\sqrt{n}}$. From Lemma \ref{un4lemma2},
$\eta_{1}^{2}(\mathcal{H}_{1})-\dfrac{(n-2)^2}{n}=\dfrac{1}{2n(n-1)}\left(\sqrt{A}-(n-4)(n^2-2n+3)\right),$ where $A=n^6-12n^5+58n^4-108n^3+41n^2+40n+16$. Observe that $A-(n-4)^{2}(n^{2}-2n+3)^2=16(n-1)^2(3n-8)>0$ for $n\ge 5$,  and  so $\sqrt{A}> (n-4)(n^{2}-2n+3)$. Hence  $\eta_{1}(\mathcal{H}_{1})>\dfrac{n-2}{\sqrt{n}}$. This completes the proof of the theorem.
\end{proof}
\section{$ABS$ spectral radius of bicyclic graphs}\label{bg}
 The $\infty$-graph $B_{\infty}(p,l,q)$ \cite{guo2005spectral} is obtained from two cycles $C_{p}$ and $C_{q}$, and a path $P_{l}$ by identifying one end  of the path with a vertex of $C_{p}$ and the other end with a  vertex of $C_{q}$. The $\theta$-graph $B_{\theta}(p,l,q)$ is obtained from three vertex-disjoint paths $P_{p+2}, P_{l+2}$ and $P_{q+2}$, where $l, p, q \ge 0$ and at most one of them is $0$, by identifying the three initial vertices and terminal vertices of them, respectively. See Fig.  \ref{bicyclic1}.
 \begin{figure}[H]
\includegraphics[width=0.95\textwidth]{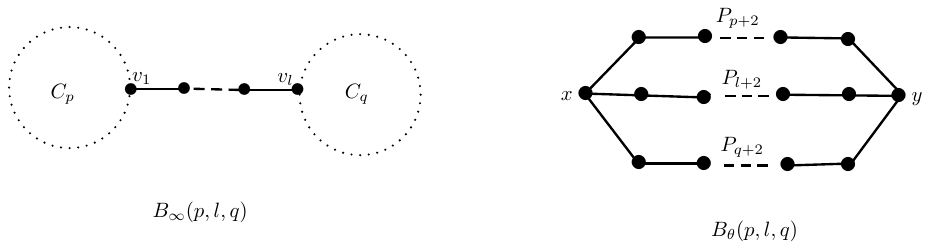}
\caption{Graphs $B_{\infty}(p,l,q)$ and $B_{\theta}(p,l,q)$. }
\label{bicyclic1}
\end{figure}
Let $\mathcal{B}(n)$ be the class of all bicyclic graphs of order $n$. For a bicyclic graph $G$, the unique bicyclic subgraph of $G$ that contains no pendant vertices is denoted by $\mathbb{U}(G)$. Note that $\mathbb{U}(G)$ is either $B_{\infty}(p,l,q)$ or $B_{\theta}(p,l,q)$. 
Denote 
\begin{center}
 $\mathcal{B}_{\infty}(n)=\left\{G\in \mathcal{B}_n| \mathbb{U}(G)=B_{\infty}(p,l,q)\right\}$ and $\mathcal{B}_{\theta}(n)=\left\{G\in \mathcal{B}(n)| \mathbb{U}(G)=B_{\theta}(p,l,q)\right\}$.
\end{center} Then $\mathcal{B}(n)=\mathcal{B}_{\infty}(n)\cup \mathcal{B}_{\theta}(n)$.
The following Kelmans operation is very much essential in this section.\\[2mm]
Let $u$ and $v$ be two vertices of the graph $G$. Define $\Omega_1=N(u)-N[v]$, $\Omega_2=N(v)-N[u]$, $\Omega_3=N(u)\cap N(v)$.  The Kelmans operation on two vertices $u$ and $v$ results in a new graph, denoted by $G_{uv^{+}}$. It is  obtained by deleting all the edges $uw$ such that $w \in \Omega_{1}$, and then adding edges $vw$ for all such $w \in \Omega_{1}$.\\[2mm] The following lemma proved in \cite{shetty2024mathcal1} gives a comparison between $ABS$ spectral radii of $G$ and $G_{uv^{+}}$, respectively.

\begin{lemma}{\rm\cite{shetty2024mathcal1}}\label{emptydelta}
 Let $G$ be a connected graph of order $n$. Suppose $uv\in E(G)$ such that $|\Omega_1|\ge 1$, $|\Omega_{2}|\ge 1$ and $\Omega_3=\emptyset$. Then $\eta_1(G)<\eta_1(G_{uv^{+}})$.
\end{lemma}
Let $\Gamma_\infty(n) \subseteq \mathcal{B}_\infty(n)$ denote the class of bicyclic graphs of order $n$ with girth $3$, diameter at most $4$, and exactly $n - 5$ pendant vertices. Also, let $\Gamma_{\theta}(n) \subseteq \mathcal{B}_{\theta}(n)$ denote the class of bicyclic graphs of order $n$ with girth $3$, diameter at most $4$, and exactly $n - 4$ pendant vertices. See Fig. \ref{biafkel}.
\begin{figure}[H]
~~~~~~~~~~~~\includegraphics[width=0.75\textwidth]{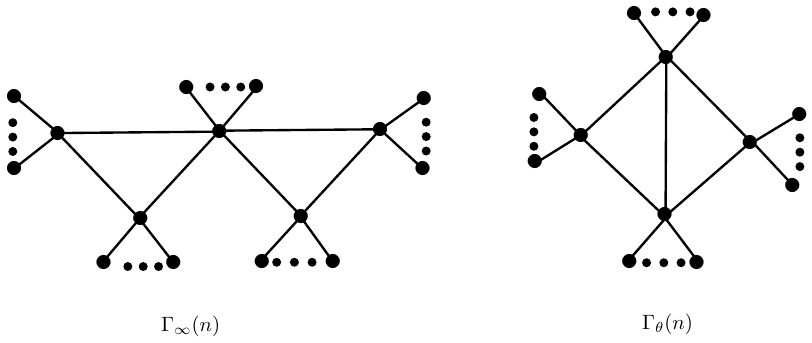}
\caption{Graph classes $\Gamma_{\infty}(n)$ and $\Gamma_{\theta}(n)$.}
\label{biafkel}
\end{figure}
\begin{theorem}\label{theorem}
Let $G\in \mathcal{B}_{\infty}(n)$, where $n\ge 5$ with the largest $ABS$ spectral radius. Then $G\in \Gamma_\infty(n)$. 
\end{theorem}
\begin{proof}
Let $G\in \mathcal{B}_{\infty}(n)$ with maximum $ABS$ spectral radius and let $\mathbb{U}(G)=B_{\infty}(p,l,q)$. Suppose that $G$ has a non-pendant edge $uv$ which does not lie on either of the two  cycles $C_{p}$ and $C_{q}$. By applying Kelmans operation on the edge $uv$, we obtain the graph $G_{uv^{+}}$. Now, Lemma \ref{emptydelta} ensures that $\eta_{1}(G)<\eta_{1}(G_{uv^{+}})$, a contradiction. Hence, every edge that does not lie on $C_{p}$ or $C_{q}$ must be pendant. Furthermore, if either  $C_{p}$ or $C_{q}$ is not a triangle, applying the operation to an edge 
$uv$ on such a non-triangular cycle results in the graph $G_{uv^{+}}$ and Lemma \ref{emptydelta} again implies  $\eta_{1}(G)<\eta_{1}(G_{uv^{+}})$, a contradiction. Therefore,  $G\in \Gamma_{\infty}(n)$.  
\end{proof}
Our next result can be proved in a manner similar to Theorem \ref{theorem}. 
\begin{theorem}
Let $G\in \mathcal{B}_{\theta}(n)$, where $n\ge 4$ with the largest $ABS$ spectral radius. Then $G\in \Gamma_\theta(n)$.
\end{theorem}

We denote the class of bicyclic graphs of order $n$ and diameter $d$ by $\mathcal{B}_n^d$. Let $P_{d+1}(i)$, where $d\ge 2$ and $2\le i\le d$ be the graph of order $n$ obtained from the path  $P_{d+1}:z_1z_2\dots z_{d+1}$ by joining the vertex $z_{i}$ with each of the isolated vertices $z_{d+2},\dots, z_n$. Denote by $P_{d+1}^\ast(i)$ the graph obtained from $P_{d+1}(i)$ by adding the edges $z_{i-1}z_n$ and $z_{i+1}z_n$, The graph $P^{\ast\ast}_{d+1}(i)$  is obtained from $P_{d+1}(i)$ by adding the edges $z_{i-1}z_{n-1}$  and $z_{i-1}z_n$. See Fig. \ref{theta} and \ref{p5p4}.
\begin{figure}[H]
~~~~~~~~~~~~~~~\includegraphics[width=0.85\textwidth]{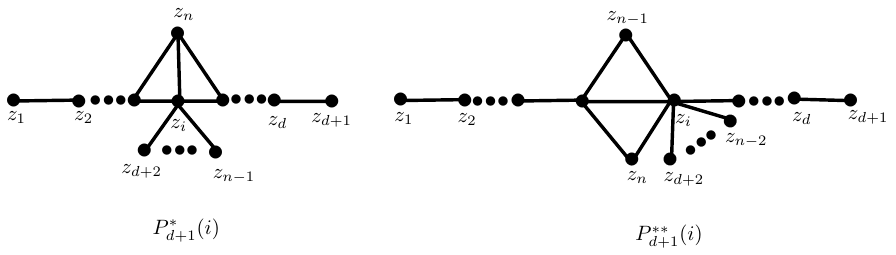}
\caption{The bicyclic graphs $P_{d+1}^*(i)$ and $P^{**}_{d+1}(i)$.}
\label{theta}
\end{figure}
The following lemma proved in \cite{guo2007spectral} is useful to prove our next result.
\begin{lemma}{\rm\cite{guo2007spectral}}\label{d3d4}
Let $n\ge d+4$ and $G\in \mathcal{B}_{n}^d$. If $d\ge 4$, then $\lambda_1(G)\le \lambda_1\left(P^{\ast}_{d+1}\left(\lfloor\dfrac{d+2}{2}\rfloor\right)\right)$ with equality if and only if $G=P^{\ast}_{d+1}\left(\lfloor \dfrac{d+2}{2}\rfloor\right)$. Furthermore, if $d=3$, then $\lambda_1(G)\le \lambda_1(P^{\ast\ast}_4(3))$ with equality if and only if $G=P^{\ast\ast}_4(3)$.
\end{lemma}
\begin{figure}[H]
~~~~~~\includegraphics[width=0.85\textwidth]{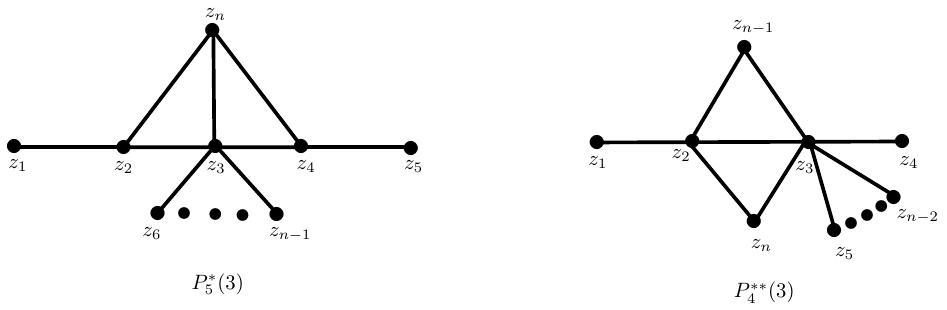}
\caption{ The bicyclic graphs $P_5^*(3)$ and $P_4^{**}(3)$.}
\label{p5p4}
\end{figure}
We now derive bounds for the $ABS$ spectral radius of graphs with a given diameter that belong to either $\Gamma_\infty(n)$ or $\Gamma_\theta(n)$.  
\begin{lemma}\label{bigirth4}
Let $G$ be a graph in the class $\Gamma_\infty(n)$ or $\Gamma_\theta(n)$. If diameter of $G$ is $4$, then 
 $\eta_1(G)<\sqrt{\dfrac{(n-2)(n-1)}{n}}$ for all $n\ge 16$.
\end{lemma}
\begin{proof}
For any two vertices $v_{i}$ and $v_{j}$ of $G$, $d_{j}+d_{j}\le n$. Therefore, $\mathcal{ABS}(G)\preceq\sqrt{\dfrac{n-2}{n}}A(G)$, and by Lemma \ref{b_1b_2}, $\eta_{1}(G)<\sqrt{\dfrac{n-2}{n}}\lambda_{1}(G)$. Furthermore, by Lemma \ref{d3d4}, \begin{equation}\label{plemma}\eta_{1}(G)<\sqrt{\dfrac{n-2}{n}}\,\lambda_{1}(P^\ast_5(3))
\end{equation} for $n\ge 8$.\\[2mm] We now show that $\lambda_{1}(P^\ast_5(3))< \sqrt{n-1}$ for all $n\ge 16$. The adjacency matrix of $P^\ast_5(3)$ is as follows.\\[2mm]
$$ \mathcal{A}(P^\ast_5(3))=\begin{bmatrix}
	0&1&0&0&0&0&\dots&0&0\\
    1&0&1&0&0&0&\dots&0&1\\
    0&1&0&1&0&1&\dots&1&1\\
    0&0&1&0&1&0&\dots&0&1\\
    0&0&0&1&0&0&\dots&0&0\\
    0&0&1&0&0&0&\dots&0&0\\
    \vdots&&&&&&\ddots&&\vdots\\
    0&0&1&0&0&0&\dots&0&0\\
    0&1&1&1&0&0&\dots&0&0
\end{bmatrix}.$$
Therefore, the characteristic polynomial of $A(P^\ast_5(3))$ is $x^{n-6}(x^2-1)\varphi(x)$, where $\varphi(x)=x^4-nx^2-4x+3n-17.$ It is straightforward to verify that $\varphi(\sqrt{n-1})>0$ for all $n\ge 16$.\\ Let $n\ge 16$ and  $\varphi_1(x)=\varphi(x+\sqrt{n-1})$. Then $\varphi_1'(x)=\left( 12\,{x}^{2}+2\,n-4 \right) \sqrt {n-1}+4\,{x}^{3}+ \left( 10
\,n-12 \right) x-4
>0$ for all $x\ge 0$. Thus, $\varphi(x)$
is increasing for $x\ge\sqrt{n-1}$. So, $\varphi(x)\ge \varphi(\sqrt{n-1})>0$ for $x\ge\sqrt{n-1}$. This implies, $\varphi(x)$ has no root in the interval $[\sqrt{n-1},\infty)$.
This proves that for $n\ge 16$,  \begin{equation}\label{plemma1}
\lambda_1(P^\ast_5(3))<\sqrt{n-1}.	
\end{equation}   
Employing equation (\ref{plemma1}) in (\ref{plemma}), we get the desired result.
\end{proof}
\begin{lemma}\label{bigirth3}
Let $G$ be a graph  in the class  $\Gamma_\infty(n)$ or $\Gamma_\theta(n)$. If diameter of $G$ is $3$, then \begin{enumerate}[(i)]\item $\eta_1(G)<\dfrac{1}{10}\sqrt{\dfrac{(n-1)(100n-53)}{n+1}}$ for all $n\ge 21$.
\item $\eta_1(G)<\dfrac{n-1}{\sqrt{n+1}}$ for all $n\ge 34$.\end{enumerate}
\end{lemma}
\begin{proof}
For any two vertices $v_{i}$ and $v_{j}$ of $G$, $d_{i}+d_{j}\le n+1$. Therefore,  $\mathcal{ABS}(G)\preceq\sqrt{\dfrac{n-1}{n+1}}A(G)$. Now, by Lemmas \ref{b_1b_2} and \ref{d3d4},   $\eta_{1}(G)<\sqrt{\dfrac{n-1}{n+1}}\lambda_{1}(G)\le \sqrt{\dfrac{n-1}{n+1}}\lambda_{1}(P^{\ast\ast}_4(3))$ for all $n\ge 7$. That is,   
\begin{equation}\label{pl1}
\eta_{1}(G)<\sqrt{\dfrac{n-1}{n+1}}\lambda_{1}(P^{\ast\ast}_4(3)).
\end{equation}
We now show that $\lambda_{1}(P^{\ast\ast}_4(3))< \dfrac{1}{10}\sqrt{100n-53}$ for all $n\ge 21$ and $\lambda_{1}(P^{\ast\ast}_4(3))< \sqrt{n-1}$  for all $n\ge 34$. The adjacency matrix of $P^{\ast\ast}_4(3)$ is as follows.\\[2mm]
$$A(P^{\theta}_4(3))=\begin{bmatrix}
	0&1&0&0&0&\dots&0&0&0\\
    1&0&1&0&0&\dots&0&1&1\\
    0&1&0&1&1&\dots&1&1&1\\
    0&0&1&0&0&\dots&0&0&0\\
    0&0&1&0&0&\dots&0&0&0\\
    \vdots&&&&&\ddots&&&\vdots\\
    0&0&1&0&0&\dots&0&0&0\\
    0&1&1&0&0&\dots&0&0&0\\
    0&1&1&0&0&\dots&0&0&0
\end{bmatrix}.$$
Therefore, the characteristic polynomial of $A(P^{\ast\ast}_4(3))$ is $x^{n-4}\varrho(x)$, where $\varrho(x)=x^4-(n+1)x^2-4x+3n-13$. It is straightforward to verify that $\varrho(\dfrac{1}{10}\sqrt{100n-53})>0$ for all $n\ge 21$ and $\varrho(\sqrt{n-1})>0$ for all $n\ge 34$. Let $n\ge 21$ and  $\varrho_{1}(x)=\varrho(x+\dfrac{1}{10}\sqrt{100n-53})$. Then
\[\varrho_1^{\prime}(x)={\frac {1}{250}}\, \left( 300\,{x}^{2}+50\,n-103 \right) \sqrt {100\,n
	-53}-4+4\,{x}^{3}+{\frac {1}{250}}\, \left( 2500\,n-2090 \right) x>0\] for all $x\ge 0.$ This implies, $\varrho(x)$ is increasing for $x\ge \dfrac{1}{10}\sqrt{100n-53}$. So, $\varrho(x)\ge\varrho(0)>0$ for all $x\ge \dfrac{1}{10}\sqrt{100n-53}$. Hence, $\varrho(x)$ has no root in the interval $[\dfrac{1}{10}\sqrt{100n-53},\infty)$. Thus, for $n\ge 21$,
\begin{equation}\label{pl2}
\eta_{1}(P_{4}^{\ast\ast}(3))<\dfrac{1}{10}\sqrt{100n-53}.	
\end{equation}
 Now. let $n\ge 34$ and $\varrho_2(x)=\varrho(x+\sqrt{n-1})$. Then \[\varrho_2'(x)=\left( 12\,{x}^{2}+2\,n-6 \right) \sqrt {n-1}-4+4\,{x}^{3}+ \left( 10
 \,n-14 \right) x
 >0\] for all $x\ge 0$. This implies, $\varrho(x)$
is increasing for $x>\sqrt{n-1}$. So, $\varrho(x)$ has no root in the interval $[\sqrt{n-1},\infty)$. Thus, for $n\ge 34$, \begin{equation}\label{pl3}\lambda_1(P^{\ast\ast}_4(3))<\sqrt{n-1}.
\end{equation} 
From equations (\ref{pl1}), (\ref{pl2}) and (\ref{pl3}), we arrive at the result.
\end{proof}

\begin{figure}[H]
~~~~~~~~~~~~~~~~~~~~~~\includegraphics[width=0.55\textwidth]{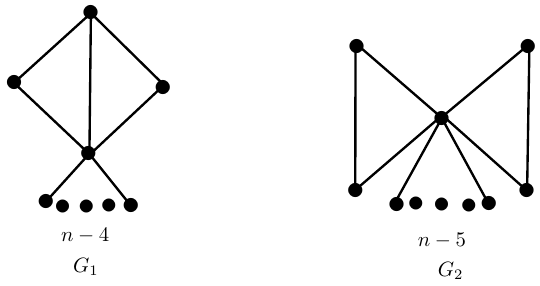}
\caption{Bicyclic graphs $G_1$ and $G_2$.}
\label{bicyclic2}
\end{figure}
\begin{lemma}\label{big2}
Let $G_{2}$ be the graph as depicted in Fig. \ref{bicyclic2}. Then \begin{enumerate}[(i)]\item $\dfrac{1}{10}\sqrt{\dfrac{(n-1)(100n-53)}{n+1}}<\eta_1(G_2)<\dfrac{1}{10}\sqrt{\dfrac{(n-1)(100n-53)}{n+1}}+\dfrac{1}{2n}$, the left side inequality is valid for $ 5\le n\le 41$ and right side inequality is valid for all $n\ge 22$.
\item $\dfrac{n-1}{\sqrt{n+1}}<\eta_1(G_2)<\dfrac{n-1}{\sqrt{n+1}}+\dfrac{3}{2n}$, the left side inequality valid is  for $n\ge 5$ and right side inequality is valid for all $n\ge 34$. \end{enumerate}
\end{lemma}
\begin{proof}
The $ABS$ matrix of $G_2$ is as follows.

$$ {\small\begin{bmatrix}
		0& \alpha&\alpha&\alpha&\alpha&\beta&\dots&\beta\\[2mm]
		\alpha&0&\sqrt{\dfrac{1}{2}}&0&0&0&\dots&0\\[2mm]
		\alpha&\sqrt{\dfrac{1}{2}}&0&0&0&0&\dots&0\\[2mm]
		\alpha&0&0&0&\sqrt{\dfrac{1}{2}}&0&\dots&0\\[2mm]
		\alpha&0&0&\sqrt{\dfrac{1}{2}}&0&0&\dots&0\\[2mm]
		\beta&0&0&0&0&0&\dots&0\\[2mm]
		\vdots&&&&&\vdots&\ddots&\vdots\\[2mm]
		\beta&0&0&0&0&0&\dots&0\\
	\end{bmatrix},}$$

 where $\alpha=\sqrt{\dfrac{n-1}{n+1}}$ and $\beta=\sqrt{\dfrac{n-2}{n}}$. Therefore, characteristic polynomial of $\mathcal{ABS}(G_2)$ is $\dfrac{x^{n-6}}{2n(n+1)}\left(x+\dfrac{1}{\sqrt{2}}\right)^2(x-\dfrac{1}{\sqrt{2}})\psi(x)$, where $\psi(x)= 2n(n+1) {x}^{3}-\sqrt{2}n(n+1){x}^{2}-2\left( {n}^{3}-2\,{n}^{2}-n
+10 \right) x+\sqrt {2}({n}^{3}-6{n}^{2}+3n
+10).$\\[2mm]
Note that, $\psi(\dfrac{1}{10}\sqrt{\dfrac{(n-1)(100n-53)}{n+1}})<0$ for $5\le n\le 41$ and $\psi(\sqrt{n})>0$.
Also, $\psi(\dfrac{1}{10}\sqrt{\dfrac{(n-1)(100n-53)}{n+1}}+\dfrac{1}{2n})>0$ for all $n\ge 22$,  and $\psi(x)$ is increasing for $x\ge \dfrac{1}{10}\sqrt{\dfrac{(n-1)(100n-53)}{n+1}}+\dfrac{1}{2n}$. 
Thus, $\eta_{1}(G_{2})>\dfrac{1}{10}\sqrt{\dfrac{(n-1)(100n-53)}{n+1}}$ for $5\le n\le 41$, and $\eta_{1}(G_{2})<\dfrac{1}{10}\sqrt{\dfrac{(n-1)(100n-53)}{n+1}}+\dfrac{1}{2n}$ for all $n\ge 22$.\\[2mm]
We also find that,
 $\psi\left(\dfrac{n-1}{\sqrt{n+1}}\right)<0$ for $n\ge 5$. Moreover,  $\psi\left(\dfrac{n-1}{\sqrt{n+1}}+\dfrac{3}{2n}\right)>0$ for $n\ge 34$ and $\psi(x)$ is increasing for $x\ge  \dfrac{n-1}{\sqrt{n+1}}+\dfrac{3}{2n}$. Thus, $\eta_{1}(G_{2})>\dfrac{n-1}{\sqrt{n+1}}$ for $n\ge 5$ and  $\eta_1(G_2)<\dfrac{n-1}{\sqrt{n+1}}+\dfrac{3}{2n}$, for $n\ge 34$.\\
\end{proof}

The following theorem is a direct consequence of Lemmas~\ref{bigirth4},~\ref{bigirth3}
 and \ref{big2}.
\begin{theorem}\label{im}
For any graph $G\in \mathcal{B}(n)$ of order $n\ge 21$ with $G\ncong G_1$, we have  $\eta_1(G)\le \eta_1(G_2)$. Further, equality holds if and only if $G\cong G_2$.
\end{theorem}
\begin{lemma}\label{big1}
Let $G_{1}$ be the graph as depicted in Fig. \ref{bicyclic2}. Then \begin{enumerate}[(i)]\item $\eta_1(G_1)>\dfrac{1}{10}\sqrt{\dfrac{(n-1)(100n-53)}{n+1}}+\dfrac{1}{2n}$ for $5\le n\le 33$.
\item $\eta_1(G_1)>\dfrac{n-1}{\sqrt{n+1}}+\dfrac{3}{2n}$ for $n\ge 5$.\end{enumerate}
\end{lemma}
\begin{proof}
The $ABS$ matrix of $G_1$ is as follows.
$$\begin{bmatrix}
	0& \alpha&\gamma&\alpha&\beta&\dots&\beta\\
	\alpha&0&\sqrt{\dfrac{3}{5}}&0&0&\dots&0\\
	\gamma&\sqrt{\dfrac{3}{5}}&0&\sqrt{\dfrac{3}{5}}&0&\dots&0\\
	\alpha&0&\sqrt{\dfrac{3}{5}}&0&0&\dots&0\\
	\beta&0&0&0&0&\dots&0\\
	\vdots&&&&\vdots&\ddots&\vdots\\
	\beta&0&0&0&0&\dots&0
\end{bmatrix}.$$
where $\alpha=\sqrt{\dfrac{n-1}{n+1}}$, $\beta=\sqrt{\dfrac{n-2}{n}}$, and $\gamma=\sqrt{\dfrac{n}{n+2}}$. 
Therefore, characteristic polynomial of $\mathcal{ABS}(G_1)$ is $x^{n-4}\phi(x)$,
where \[\phi(x)={x}^{4}-\dfrac{1}{5}\,{\dfrac { \left( 5\,{n}^{4}+6\,{n}^{3}-7\,{n}^{2}+52
		\,n+80 \right) {x}^{2}}{n \left( n+2 \right)  \left( n+1
		\right) }}-\dfrac{4}{5}\,\sqrt {{\frac {n(n-1)15}{(n+2)(n+1)}}}{x}+\dfrac{6}{5}\,{\frac { \left( n-2 \right)  \left( n-4
		\right)}{n}}.\]
Since $\phi(x)$ is a monic polynomial with real roots, and  \[\phi(\dfrac{1}{10}\sqrt{\dfrac{(n-1)(100n-53)}{n+1}}+\dfrac{1}{2n})<0\,\, (5\le n\le 33)\] and
 
\[\phi\left(\dfrac{n-1}{\sqrt{n+1}}+\dfrac{3}{2n}\right)<0\,\,(n\ge5),\] the desired upper bounds follow.
\end{proof}
\begin{figure}[H]
~~~~~~~~~~~~~\includegraphics[width=0.5\textwidth]{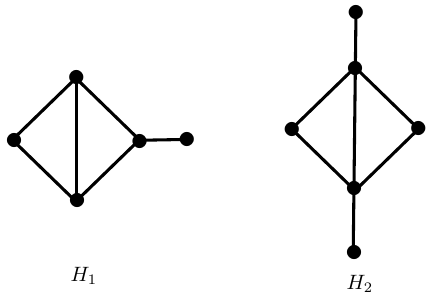}
 \caption{Bicyclic graphs $H_1$ and $H_2$.}
 \label{H1}
 \end{figure}
 \begin{theorem}
 Let $G\in \mathcal{B}(n)$ with $n\ge 5$. Let $H_{1}$ and $H_{2}$ be the graphs as depicted in Fig. {\rm\ref{H1}}.
 \begin{enumerate}[(i)]
 \item If $n=5$, then $G_1$ and $H_1$ are, respectively, the unique graphs with the first two maximum $ABS$ spectral radii, equal to $2.1637$ and $2.1023$.
 \item If $n = 6$, then $G_1$ and $H_2$ are, respectively, the unique graphs with the first two
maximum $ABS$ spectral radii, equal to $2.3220$ and $2.2915$.
 \item If $n\ge 7$, then $G_1$ and $G_2$ are, respectively, the unique graphs with the first two
maximum $ABS$ spectral radii.
 \end{enumerate}\end{theorem}
\begin{proof}
Let $n\ge 22$. By Theorem~\ref{im}, Lemma~\ref{big2} and \ref{big1}, $G_1$ and $G_2$ are, respectively, the unique graphs with the first two
maximum $ABS$ spectral radii. For $7\le n\le 21$, it can be verified using Sage \cite{Sage} that $G_1$ has the largest and $G_2$ attains the second largest $ABS$ spectral radius. Similarly, for the case $n=5$ or $6$, we get the desired graphs.
\end{proof}
 \section*{Conclusion}
In this work, we identified the unique bipartite unicyclic graph with the maximum $ABS$ spectral radius. We also determined the unique bicyclic graphs attaining the largest and the second largest $ABS$ spectral radii. These results contribute to the broader problem of characterizing extremal graphs with respect to the $ABS$ spectral radius within specific classes of graphs.
\section*{Declarations}
\textbf{Availability of data and material}: Manuscript has no associated data.\\
\textbf{Competing interest}: The authors declare that they have no competing interest.\\
\textbf{Funding}: No funds received.\\
\textbf{Authors' contributions:} All authors have equally contributed.

\bibliography{name}
\bibliographystyle{abbrv}
\end{document}